\documentclass[a4paper, 12pt]{article} 

\usepackage[utf8]{inputenc}
\usepackage[T1]{fontenc}
\usepackage[a4paper, margin=2.75cm]{geometry}
\usepackage{needspace, mathscinet}
\usepackage{blkarray} 
\usepackage{multirow}
\usepackage{libertine} 
\usepackage{inconsolata} 

\usepackage{authblk, amsmath, amsthm, amssymb, graphicx, thmtools, thm-restate, mathtools}
\usepackage[textsize=small, textwidth=2cm, color=yellow]{todonotes}
\usepackage[colorlinks=true, citecolor=blue, linkcolor=blue, urlcolor=blue]{hyperref}
\usepackage{enumitem}
\usepackage[noabbrev,capitalize,nameinlink]{cleveref}
\crefname{claim}{Claim}{Claims}
\crefname{lemma}{Lemma}{Lemmas}
\crefname{theorem}{Theorem}{Theorems}
\crefname{proposition}{Proposition}{Propositions}
\crefname{definition}{Definition}{Definitions}
\crefname{conjecture}{Conjecture}{Conjectures}
\crefname{corr}{Corollary}{Corollaries}

\declaretheorem[name=Theorem]{theorem}
\declaretheorem[name=Lemma, sibling=theorem]{lemma}

\declaretheorem[name=Conjecture, sibling=theorem]{conjecture}
\declaretheorem[name=Claim, parent=theorem]{claim}

\def\cqedsymbol{\ifmmode$\lrcorner$\else{\unskip\nobreak\hfil
\penalty50\hskip1em\null\nobreak\hfil$\lrcorner$
\parfillskip=0pt\finalhyphendemerits=0\endgraf}\fi}

\let\le\leqslant

\let\leq\leqslant
\let\geq\geqslant

\let\OLDthebibliography\thebibliography
\renewcommand\thebibliography[1]{
  \OLDthebibliography{#1}
  \setlength{\parskip}{0pt}
  \setlength{\itemsep}{0pt plus 0.3ex}
}%
\AtBeginDocument{
   \def\MR#1{}
}%

\newcommand{\GF}{\textsf{GF}}

\title{Girth in $\GF(q)$-representable matroids}

\author{James Davies}
\affil{Trinity Hall, Cambridge, UK.}
\author{Meike Hatzel\thanks{Supported by the Institute for Basic Science (IBS-R029-C1).}}
\affil{Discrete Mathematics Group, Institute for Basic Science (IBS), Daejeon, South Korea.}
\author{Kolja Knauer\thanks{Supported by the Ministerio de Ciencia, Innovación y Universidades Grant No.~PID2022-137283NB-C22}}
\affil{Departament de Matem\'{a}tiques i Inform\'{a}tica, Universitat de Barcelona, Spain.}
\author{Rose McCarty\thanks{Supported by the National Science Foundation under Grant No.~DMS-2202961.}}
\affil{School of Mathematics and School of Computer Science, Georgia Institute of Technology, USA.}
\author{Torsten Ueckerdt}
\affil{Institute of Theoretical Informatics, Karlsruhe Institute of Technology, Germany.}
\begin{document}

\maketitle

\begin{abstract}
    We prove a conjecture of Geelen, Gerards, and Whittle that for any finite field $\GF(q)$ and any integer $t$, every cosimple $\GF(q)$-representable matroid with sufficiently large girth contains either $M(K_t)$ or $M(K_t)^*$ as a minor.
\end{abstract}

\section{Introduction}

The \emph{girth} of a matroid $M$ is the minimum number of elements in a circuit of $M$, or $\infty$ if $M$ has no circuits. 
Examples of cosimple matroids with large girth include the graphic matroid of a $3$-edge-connected graph with large girth and $M(K_t)^*$, the dual of the graphic matroid of the $t$-vertex clique $K_t$.
Geelen, Gerards, and Whittle~\cite[Conjecture~5.4]{highlyConnMatroids} conjectured that every cosimple $\GF(q)$-representable matroid of large girth contains one of these examples as a minor. We prove their conjecture.

\begin{restatable}{theorem}{girthThm}
\label{thm:main}
    For any finite field $\GF(q)$ and any integer $t$, there exists an integer $f(t,q)$ such that every cosimple $\GF(q)$-representable matroid with girth at least $f(t,q)$ contains either $M(K_t)$ or $M(K_t)^*$ as a minor.
\end{restatable}

\Cref{thm:main} generalizes the theorem of Thomassen~\cite{ThomassenGirthGraphs} that any graph of minimum degree at least three and sufficiently high girth contains $K_t$ as a minor. Thomassen's Theorem is celebrated, and there are several strengthenings known for graphs~\cite{KuhnOsthus2003, Mader98}. By considering the cographic case, we can see that \cref{thm:main} also generalizes the classic lemma due to Mader~\cite{MaderAverageDegree} (and optimized by Thomason~\cite{Thomason84} and Kostochka~\cite{Kostochka82, Kostochka84}) which says that any sufficiently dense graph contains $K_t$ as a minor. Both graphic and cographic matroids must be included as potential outcomes in \cref{thm:main}; this is because not every graph has a small cycle or cut.
However, perhaps the condition about representativity could be relaxed. We discuss this possibility in \cref{sec:conclusion}.


It was believed that a proof of \cref{thm:main} would require the use of a structure theorem for matroid minors~\cite{highlyConnMatroids}. Yet our proof of \cref{thm:main} is surprisingly short; it relies on a previously unexplored connection between the Matroid Growth Rate Theorem~\cite{geelen2003cliques} and Haussler's Shallow Packing Lemma~\cite{Haussler95}.
In the context of simple $\GF(q)$-representable matroids with a forbidden graphic minor, the Growth Rate Theorem of Geelen and Whittle~\cite{geelen2003cliques} bounds the number of elements of the matroid by a linear function of its rank (we remark that there is also a more general Growth Rate Theorem of Geelen, Kung, and Whittle~\cite{GrowthRateTheorem}).
This theorem directly generalizes Mader's Theorem~\cite{MaderAverageDegree}.





Our key observation is that for $\GF(q)$-representable matroids, the Growth Rate Theorem (\cref{thm:GFgrowthRate}) can be interpreted in terms of the shatter function of an associated set system.
This observation allows us to apply powerful tools such as Haussler's Shallow Packing Lemma~\cite{Haussler95} (\cref{lem:packing}). These concepts are fundamental notions in discrete and computational geometry~\cite{MatousekBook, Suk24, Welzl88}, the combinatorics of set systems~\cite{Sauer1972, Shelah1972, VC15}, and first-order logic~\cite{complexityTWW, flipperGamesMonStable, RVS19}. However, \cref{thm:main} is the first application we know of to matroids\footnote{A special case of this connection was implicitly used by the fourth author in order to motivate a conjecture~\cite[Conjecture~3.5.4]{McCartyThesis} about the neighborhood complexity of graphs with a forbidden vertex-minor.}.

The associated set system we consider in order to apply Haussler's Shallow Packing Lemma is inspired by fundamental graphs. If $M$ can be represented by the columns of a binary matrix $\big[ I\text{ | }A \big]$ where $I$ is an identity matrix whose columns correspond to a basis $B$ of $M$, then the \emph{fundamental graph} with respect to $B$ is the bipartite graph whose bipartite adjacency matrix is $A$. Thus, binary matroids are determined by their fundamental graphs. Fundamental graphs were originally defined for binary matroids~\cite{graphicIsoSystems, RWAndVM}. However, they can also be defined and used for general matroids~\cite{GGKexcluded, GHMO}. 
However, this usually requires more care, since matroids are not generally determined by their fundamental graphs. Instead of taking this approach, we will define an associated set system that stores more information about the matroid than its fundamental graph.

\section{The proof}
\label{sec:girth}
In this section, we prove \cref{thm:GFqCase}, which immediately implies \cref{thm:main}.
First, we need to introduce some notation, as well as the Growth Rate Theorem and Haussler's Shallow Packing Lemma. 



The Growth Rate Theorem for $\GF(q)$-representable matroids of Geelen and Whittle~\cite{geelen2003cliques} says the following.
We remark that Nelson, Norin, and Omana \cite{nelson2023density} have recently improved the bounds on $\ell(t,q)$ to a singly exponential function.

\begin{theorem}[\cite{geelen2003cliques}]
\label{thm:GFgrowthRate}
    For any integers $t$ and $q$, there exists an integer $\ell(t,q)$ such that any simple rank-$n$ $\GF(q)$-representable matroid with no $M(K_t)$ minor has at most $\ell(t,q)\cdot n$ elements.
\end{theorem}

To state Haussler's Shallow Packing Lemma, we need to introduce some definitions. Given a finite ground set $V$, a \emph{set system} $\mathcal{F}$ on $V$ is a subset of $2^V$. (We do not allow multisets.) The \emph{shatter function} of $\mathcal{F}$, denoted by $\pi_{\mathcal{F}}(m)$, is the maximum size of $\mathcal{F}$ when restricted to any $m$ elements in $V$. That is, $\pi_{\mathcal{F}}(m)$ is the maximum, over all $m$-element subsets $W\subseteq V$, of the number of equivalence classes of the relationship $\sim_{W}$ on $\mathcal{F}$ where two sets $F,F' \in \mathcal{F}$ satisfy $F\sim_{W}F'$ if $F \cap W = F'\cap W$. For a positive integer $\delta$, we say that two sets $F,F' \in \mathcal{F}$ are \emph{$\delta$-separated} if their symmetric difference $F \Delta F'$ has size at least $\delta$ (that is, there are at least $\delta$ elements in $V$ which are in one of $F,F'$ but not the other). We say that $\mathcal{F}$ is \emph{$\delta$-separated} if any pair of distinct elements in $\mathcal{F}$ are $\delta$-separated.


We use the following version of Haussler's Shallow Packing Lemma~\cite{Haussler95}. This version is stated as~\cite[Lemma~2.2]{FPS19}, for instance. (Actually~\cite[Lemma~2.2]{FPS19} is a more general version; we only require the case that $d=1$.)

\begin{lemma}[\cite{Haussler95}]
\label{lem:packing}
For any number $\ell \geq 1$, there exists an integer $c=c(\ell)$ so that for every positive integer $\delta$, if $\mathcal{F}$ is a set system on a finite ground set $V$ so that $\mathcal{F}$ is $\delta$-separated and $\pi_\mathcal{F}(m)\leq \ell m$ for every positive integer $m$, then $|\mathcal{F}| \leq c|V|/\delta$.
\end{lemma}


The following theorem immediately implies \cref{thm:main}.
We remark that even in the context of graphs (so when $M$ is graphic), this theorem already provides another strengthening of Thomassen's theorem~\cite{ThomassenGirthGraphs}.

\begin{theorem}
\label{thm:GFqCase}
    For any finite field $\GF(q)$ and any integer $t$, there exists an integer $k=k(t,q)$ such that if $M$ is a cosimple $\GF(q)$-representable matroid not containing $M(K_t)$ or $M(K_t)^*$ as a minor, then for every basis $B$ of $M$, there is a circuit $C$ of size at most $k$ with $|C\backslash B|\le 2$.
\end{theorem}
\begin{proof}
Let $M$ be a cosimple $\GF(q)$-representable matroid which does not contain $M(K_t)$ or $M(K_t)^*$ as a minor, and let $B$ be a basis of $M$. By performing row operations and deleting all zero rows, we can obtain a matrix over $\GF(q)$ of the form $\big[ I\text{ | }A \big]$ so that $I$ is a $|B| \times |B|$ identity matrix, and $M$ is represented by the column vectors of $\big[ I\text{ | }A \big]$ so that the columns of $I$ correspond to the elements in $B$. Thus we may view $A$ as a $|B| \times |E(M)\setminus B|$-matrix.
Given elements $b \in B$ and $e \in E(M)\setminus B$, we write $A_{b,e}$ for the corresponding entry of~$A$.


Recall that in the binary case, the fundamental graph is the bipartite graph with adjacency matrix $A$. In this case, the set system would consist of the supports of the columns of $A$. In the general case, we need to store more information about \emph{which} element of $\GF(q)$ is contained in an entry $A_{b,e}$ of $A$. So we now define a set system $\mathcal{F}$ on the ground set $B \times (\GF(q)\setminus \{0\})$, which we denote by $V$ for short. 

\begin{figure}
    \centering
    \includegraphics{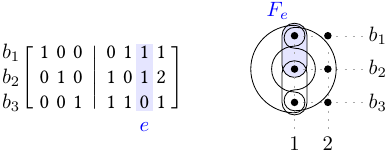}
    \caption{A ternary matroid $M$ with a representation $\big[ I\text{ | }A \big]$ over $\GF(3)$ and the corresponding set system $\mathcal{F}$ with $e \in E(M)\setminus B$ and $F_e = \{(b_1,1), (b_2,1)\}$ highlighted in blue.}
    \label{fig:ternary-example}
\end{figure}

For each element $e \in E(M)\setminus B$, we write $F_e$ for the set of all tuples $(b,\alpha)\in V$ so that $A_{b, e}=\alpha$; see \cref{fig:ternary-example} for an example. Then we set $\mathcal{F}=\{F_e:e \in E(M)\setminus B\}$. We may assume that all distinct elements $e,e' \in E(M)\setminus B$ have $F_e \neq F_{e'}$, since otherwise $e$ and $e'$ are represented by the same column vector in $A$, and we have found the desired circuit.

Now we prove a key claim. Let $\ell = \ell(t,q)$ be the integer from the Growth Rate Theorem (\cref{thm:GFgrowthRate}). So any simple rank-$n$ $\GF(q)$-representable matroid with no 
$M(K_t)$ minor has at most $\ell\cdot n$ elements.

\begin{claim}
\label{clm:shatterFunc}
For any positive integer $m$, we have $\pi_{\mathcal{F}}(m) \leq \ell q \cdot m$.
\end{claim}
\begin{proof}
Let $W \subseteq V$ be an $m$-element set. 

Let $B_W$ be the projection of $W$ onto $B$. That is, $B_W$ is the set of all $b \in B$ such that there exists $\alpha \in \GF(q)\setminus \{0\}$ so that $(b, \alpha) \in W$. Thus $|B_W|\leq m$. Consider taking the matrix $\big[ I\text{ | }A \big]$ which represents $M$ and deleting from it the rows corresponding to elements in $B\setminus B_W$. The column matroid of this matrix is a minor of $M$; it is obtained from $M$ by contracting the elements in $B\setminus B_W$. Thus, by the Growth Rate Theorem (\cref{thm:GFgrowthRate}), its simplification (that is, the matroid obtained by removing loops and only keeping one element from each parallel class) has at most $\ell\cdot m$ elements.

Now consider two elements $e,e' \in E(M)\setminus B$ with $F_e \cap W \neq F_{e'} \cap W$. Let $(b, \alpha) \in W$ be an element in one of these sets but not the other. Then one of $A_{b,e}$ and $A_{b,e'}$ is equal to $\alpha$ and the other is not. So in particular, the columns corresponding to $e$ and $e'$ are distinct even when restricted to rows in $B_W$. Finally, let us consider what happens when we take the simplification of a $\GF(q)$-represented matroid with distinct columns. It has at most one loop for the all zero vector, and each parallel class has at most $q-1$ elements. It follows that $\pi_{\mathcal{F}}(m) \leq (q-1)(\ell\cdot m)+1 \leq \ell q \cdot m$, as desired.
\end{proof}

Next we apply Haussler's Shallow Packing Lemma (\cref{lem:packing}). We write $c = c(\ell q)$ for the function from \cref{lem:packing}, and we set $\delta = \ell qc+1$. Notice that $\delta$ is just a function of $t$ and $q$. By \cref{clm:shatterFunc} and Haussler's Lemma, either $\mathcal{F}$ is not $\delta$-separated, or $|\mathcal{F}|\leq c|V|/\delta$.

First suppose that $|\mathcal{F}|\leq c|V|/\delta$.  
Recall that 
all distinct elements in $E(M)\setminus B$ correspond to distinct sets in $\mathcal{F}$. So $|E(M)\setminus B|=|\mathcal{F}|$. Since the dual of $M$ is a simple $\GF(q)$-representable matroid with no 
$M(K_t)$ minor, the Growth Rate Theorem (\cref{thm:GFgrowthRate}) yields $|E(M^*)|\leq \ell |E(M)\setminus B|=\ell |\mathcal{F}|$. Thus\begin{align*}
|V|= (q-1)|B|\leq q|E(M^*)| \leq \ell q|\mathcal{F}|\leq \ell qc|V|/\delta.
\end{align*}So $\delta \leq \ell qc$, however we chose $\delta = \ell qc+1$, a contradiction. 

Thus $\mathcal{F}$ is not $\delta$-separated. So there exist distinct elements $e,e' \in E(M)\setminus B$ so that there are fewer than $\delta$ elements in the symmetric difference of $F_e$ and $F_{e'}$. Thus there are fewer than $\delta$ rows of $\big[ I\text{ | }A \big]$ where the columns of $e$ and $e'$ differ. Let $B' \subseteq B$ be the basis elements corresponding to those rows. 
Given an element $e  \in E(M)$, we write $\vec{e}$ for the corresponding column vector of $\big[ I\text{ | }A \big]$.
So $\vec{e}-\vec{e'}$ is in the span of $\{\vec{b}:b \in B'\}$. Thus $B' \cup \{e,e'\}$ contains a circuit of $M$. Since $|B'|\leq \delta-1$, the theorem holds with $k = \delta+1$.
\end{proof}

\section{Conclusion}
\label{sec:conclusion}

In this section, we discuss possible extensions of \cref{thm:main} that relax the condition of being $\GF(q)$-representable.

We write $U_{2,q}$ for the $q$-element line, and, more generally, $U_{t,q}$ for the uniform matroid with $q$ elements and rank $t$. That is, $U_{t,q}$ is the $q$-element matroid where the circuits are the sets of size $t+1$.
The Growth Rate Theorem of Geelen and Whittle~\cite{geelen2003cliques} also applies to matroids that forbid $U_{2,q+2}$ as a minor, rather than just to $\GF(q)$-representable matroids. (Recall that $\GF(q)$-representable matroids do not have $U_{2, q+2}$ minors; see for instance~\cite[Corollary 6.5.3]{Oxl11}.)

\begin{theorem}[\cite{geelen2003cliques}]
\label{thm:growthRate}
    For any integers $t$ and $q$, there exists an integer $\ell(t,q)$ such that any simple rank-$n$ matroid with no $U_{2,q+2}$ or $M(K_t)$ minor has at most $\ell(t,q)\cdot n$ elements.
\end{theorem}

In light of \cref{thm:main} and the Growth Rate Theorem (\cref{thm:growthRate}), it is natural to conjecture the following.


\begin{conjecture}
For any positive integer $t$, there exists an integer $p(t)$ such that every cosimple matroid with girth at least $p(t)$ contains either $U_{2,t+2}$, $U_{t,t+2}$, $M(K_t)$, or $M(K_t)^*$ as a minor. 
\end{conjecture}

For large $t$, both $U_{t,t+2}$ and $M(K_t)^*$ have large girth, while $M(K_t)$ has a cosimple minor of large girth. However, it is less satisfying to forbid the line $U_{2, t+2}$. So it is natural to ask the more general question: What are the unavoidable cosimple matroids of large girth? To frame this problem precisely, let us consider a property $\mathcal{P}$ of classes of matroids. For example, we write $\mathcal{P}_{\mathrm{girth}}$ for the property ``the class contains cosimple matroids of arbitrarily large girth''. A class of matroids $\mathcal{M}$ is \emph{minor-minimal} with respect to $\mathcal{P}$ if $\mathcal{M}$ is minor-closed, $\mathcal{M}$ has property $\mathcal{P}$, and no proper minor-closed subclass of $\mathcal{M}$ has property~$\mathcal{P}$.

As an example, the class of all graphic matroids is minor-minimal with respect to $\mathcal{P}_{\mathrm{girth}}$ due to Thomassen's Theorem~\cite{ThomassenGirthGraphs}. Likewise, the class of all cographic matroids is minor-minimal with respect to $\mathcal{P}_{\mathrm{girth}}$ due to Mader's Theorem~\cite{Mader98}. 
It is straightforward to see that the closure of all colines $U_{t,t+2}$ under minors is also minor-minimal with respect to $\mathcal{P}_{\mathrm{girth}}$.
We conjecture that classes with property $\mathcal{P}_{\mathrm{girth}}$ have a finite characterization. 

\begin{conjecture}
\label{conj:wild}
There exist a finite number of classes $\mathcal{M}_1, \mathcal{M}_2, \ldots, \mathcal{M}_k$ which are minor-minimal with respect to property $\mathcal{P}_{\mathrm{girth}}$ such that a class of matroids $\mathcal{M}$ has property $\mathcal{P}_{\mathrm{girth}}$ if and only if it contains at least one of $\mathcal{M}_1, \mathcal{M}_2, \ldots, \mathcal{M}_k$.
\end{conjecture}

Matroids are not well-quasi-ordered under minors, even though graphs famously are, as proven by Robertson and Seymour~\cite{RobertsonSeymour}. Moreover, \cref{conj:wild} is an example of ``second-level better-quasi-ordering'', and it is wide open whether graphs are second-level better-quasi-ordered under minors. See~\cite{secondLevelEP} for recent progress on this question and~\cite{infiniteCliqueMinors} for a brief discussion by Robertson and Seymour. Still, we are optimistic about \cref{conj:wild} since there are some natural matroids to forbid.

In particular, let us write $B(G)$ for the \emph{bicircular} matroid of a graph $G$. Bicircular matroids were introduced by Sim\~{o}es-Pereira~\cite{bicircular}, and we refer the reader there for definitions.
It can be proven using Thomassen's Theorem~\cite{ThomassenGirthGraphs} that the class of bicircular matroids is minor-minimal with respect to $\mathcal{P}_{\mathrm{girth}}$.
We conjecture the following, which would imply \cref{conj:wild}.

\begin{conjecture}
\label{con:general}
There exists a function $g$ such that for every integer $t$, every cosimple matroid with girth at least $g(t)$ contains either $U_{t,t+2}$, $M(K_t)$, $M(K_t)^*$, or $B(K_t)$ as a minor.
\end{conjecture}


\section*{Acknowledgements}
This work was completed at the 12th Annual Workshop on Geometry and Graphs held at the Bellairs Research Institute in February 2025. We are grateful to the organizers and participants for providing an excellent research environment. We would also like to thank Jim Geelen and Szymon Toru\'{n}czyk for helpful comments.


\bibliographystyle{abbrv}
\bibliography{matroids}

\end{document}